\documentclass[11pt]{amsart}
\usepackage{amsmath}
\usepackage{amsthm}
\usepackage{amssymb}

\usepackage{float}
\usepackage{hyperref}
\usepackage{fullpage}
\usepackage{graphicx}
\usepackage{caption}
\usepackage{subcaption}

\usepackage{color}

\newtheorem{theo}{Theorem}[section]
\newtheorem{prop}[theo]{Proposition}
\newtheorem{lemma}[theo]{Lemma}
\newtheorem{coro}[theo]{Corollary}

\newcommand{\eps}{\varepsilon}

\title{Hitting $k$ primes by dice rolls}

\author[Alon]{Noga Alon}
\address[N.~Alon]{Department of Mathematics,
Princeton University, Princeton, NJ 08544.}
\email{\href{mailto:nalon@math.princeton.edu}{nalon@math.princeton.edu}}

\author[Malinovsky]{Yaakov Malinovsky}
\address[Y.~Malinovsky]{Department of Mathematics and Statistics,
University of Maryland, Baltimore County, Baltimore, MD  21250}
\email{\href{mailto:yaakovm@umbc.edu}{yaakovm@umbc.edu}}

\author[Martinez]{Lucy Martinez}
\address[L.~Martinez]{Department of Mathematics, Rutgers University,
Piscataway, NJ 08854}
\email{\href{mailto:lucy.martinez@rutgers.edu}{lucy.martinez@rutgers.edu}}

\author[Zeilberger]{Doron Zeilberger}
\address[D.~Zeilberger]{Department of Mathematics,
Rutgers University, Piscataway, NJ 08854}
\email{\href{mailto:doronzeil@gmail.com}{doronzeil@gmail.com}}

\begin{document}

\begin{abstract}
Let $S=(d_1,d_2,d_3, \ldots )$ be an infinite sequence of rolls of
independent fair dice. For an integer $k \geq 1$, let $L_k=L_k(S)$
be the smallest $i$ so that there are $k$ integers $j \leq i$ for which
$\sum_{t=1}^j d_t$ is a prime. Therefore, $L_k$ is the random variable
whose value is the number of dice rolls required until the
accumulated sum
equals a prime $k$ times. It is known
that the expected value of $L_1$ is close to $2.43$. Here we
show that for large $k$, the expected value of $L_k$ is $(1+o(1)) k\log_e k$,
where the $o(1)$-term tends to zero as $k$ tends to infinity. We also
include some computational results about the distribution of
$L_k$ for $k \leq 100$.
\end{abstract}

\maketitle

\smallskip

\noindent
{\bf MSC2020 subject classifications}: 60C05, 11A41, 60G40.

\section{Results}
Let $S=(d_1,d_2,d_3, \ldots )$ be an infinite sequence of rolls of
independent fair dice. Thus the $d_i$ are independent, identically
distributed random variables, each uniformly distributed on the
integers $\{1,2, \ldots ,6\}$. For each $i \geq 1$ put
$s_i=\sum_{j=1}^{i} d_j$. The sequence $S$ {\em hits}  a positive
integer $x$ if there exists an $i$ so that $s_i=x$. In that case
it hits $x$ in step $i$.

For any positive integer $k$, let $L_k=L_k(S)$ be the random variable
whose value is the smallest $i$ so that the sequence $S$ hits
$k$ primes during the first $i$ steps ($\infty$ if there is no such
$i$, but it is easy to see that with probability $1$ there is such
$i$). The random variable $L_1$ is introduced and studied
in \cite{AM}, see also  \cite{MZ}, \cite{SC} for
several variants and generalizations.

Here we consider the random variable $L_k$ for larger values of
$k$, focusing on the estimate of its expectation.

\subsection{Computational results}
This article is accompanied by a Maple package {\tt PRIMESk},
available from

{\tt  https://sites.math.rutgers.edu/\~{}zeilberg/mamarim/mamarimhtml/primesk.html} \quad,

where there are also numerous output files.

Using our Maple package, we computed the following
values of the expectation of $L_k$ for $k \leq 30$.

\vspace{0.2cm}

\begin{table}[H]
\begin{tabular}{|c|c|c|c|c|c|}
\hline
$k$  & $E(L_k)$ & $k$  & $E(L_k)$ & $k$  & $E(L_k)$\\
\hline
$1$ & $2.428497914$ & $11$ & $48.14320555$ & $21$ & $106.3962997$\\ \hline
$2$ & $5.712240468$ & $12$ & $53.61351459$ & $22$ & $112.5650207$\\ \hline
$3$ & $9.498878119$ & $13$ & $59.16406655$ & $23$ & $118.7684092$\\ \hline
$4$ & $13.65059271$ & $14$ & $64.79337350$ & $24$ & $125.0081994$\\ \hline
$5$ & $18.05408931$ & $15$ & $70.50517127$ & $25$ & $131.2881683$\\ \hline
$6$ & $22.64615402$ & $16$ & $76.30284161$ & $26$ & $137.6114097$\\ \hline
$7$ & $27.42115902$ & $17$ & $82.18566213$ & $27$ & $143.9783110$\\ \hline
$8$ & $32.37752852$ & $18$ & $88.14757626$ & $28$ & $150.3859881$\\ \hline
$9$ & $37.50029903$ & $19$ & $94.17811256$& $29$ & $156.8292462$ \\ \hline
$10$ & $42.76471868$ & $20$ & $100.2648068$ & $30$ & $163.3025173$\\ \hline
\end{tabular}
\end{table}

\vspace{0.2cm}

\noindent
The table suggests that the asymptotic value of
this expectation is $(1+o(1))k \log k$, where the $o(1)$-term tends to
zero as $k$ tends to infinity, and the logarithm here and throughout the
manuscript is in the natural basis. This is confirmed in the
results stated in the next subsection
and proved in Section \ref{sec:proofs}.

%
\noindent
The value of the standard deviation of $L_k$ for $k \leq 30$ is given in the
following table.

\vspace{0.2cm}

\begin{table}[H]
\begin{tabular}{|c|c|c|c|c|c|}
\hline
$k$  & $Std(L_k)$ & $k$  & $Std(L_k)$ & $k$  & $Std(L_k)$\\
\hline
$1$ & $2.4985553$ & $11$ & $14.9184147$ & $21$ & $23.3873070$\\ \hline
$2$ & $4.2393979$ & $12$ & $15.8185435$ & $22$ & $24.0816339$\\ \hline
$3$ & $5.7679076$ & $13$ & $16.7109840$ & $23$ & $24.7769981$\\ \hline
$4$ & $7.1185391$ & $14$ & $17.6115574$ & $24$ & $25.4821834$\\ \hline
$5$ & $8.3598784$ & $15$ & $18.5197678$ & $25$ & $26.1952166$\\ \hline
$6$ & $9.5715571$ & $16$ & $19.4227324$ & $26$ & $26.9055430$\\ \hline
$7$ & $10.7618046$ & $17$ & $20.3022748$ & $27$ & $27.5997195$\\ \hline
$8$ & $11.9062438$ & $18$ & $21.1419697$ & $28$ & $28.2678482$\\ \hline
$9$ & $12.9824596$ & $19$ & $21.9329240$ & $29$ & $28.9080719$ \\ \hline
$10$ & $13.9823359$ & $20$ & $22.6771846$ & $30$ & $29.5276021$\\ \hline
\end{tabular}
\end{table}

\vspace{0.2cm}

\noindent
The value of the skewness of $L_k$ for $k \leq 30$ is given in the
following table.

\vspace{0.2cm}

\begin{table}[H]
\begin{tabular}{|c|c|c|c|c|c|}
\hline
$k$  & $ Skew(L_k)$ & $k$  & $ Skew(L_k)$ & $k$  & $ Skew(L_k)$\\
\hline
$1$ & $3.3904247$ & $11$ & $0.7569428$ & $21$ & $0.5205173$\\ \hline
$2$ & $2.1496468$ & $12$ & $0.7362263$ & $22$ & $0.5148284$\\ \hline
$3$ & $1.6420771$ & $13$ & $0.7250716$ & $23$ & $0.5134409$\\ \hline
$4$ & $1.3892778$ & $14$ & $0.7131387$ & $24$ & $0.5108048$\\ \hline
$5$ & $1.2554076$ & $15$ & $0.6939289$ & $25$ & $0.5029053$\\ \hline
$6$ & $1.1503502$ & $16$ & $0.6657344$ & $26$ & $0.4888319$\\ \hline
$7$ & $1.0474628$ & $17$ & $0.6307374$ & $27$ & $0.4707841$\\ \hline
$8$ & $0.9487703$ & $18$ & $0.5936550$ & $28$ & $0.4528198$\\ \hline
$9$ & $0.8625227$ & $19$ & $0.5601812$ & $29$ & $0.4391145$ \\ \hline
$10$ & $0.7974496$ & $20$ & $0.5351098$ & $30$ & $0.4324204$\\ \hline
\end{tabular}
\end{table}

\vspace{0.2cm}

\noindent
The value of the kurtosis of $L_k$ for $k \leq 30$ is given in the
following table.
\vspace{0.2cm}

\begin{table}[H]
\begin{tabular}{|c|c|c|c|c|c|}
\hline
$k$  & $Ku(L_k)$ & $k$  & $Ku(L_k)$ & $k$  & $Ku(L_k)$\\
\hline
$1$ & $20.6214485$ & $11$ & $3.9630489$ & $21$ & $3.4553514$\\ \hline
$2$ & $10.0475452$ & $12$ & $3.9427896$ & $22$ & $3.4675149$\\ \hline
$3$ & $7.2098904$ & $13$ & $3.9031803$ & $23$ & $3.4566369$\\ \hline
$4$ & $6.1044828$ & $14$ & $3.8308431$ & $24$ & $3.4199435$\\ \hline
$5$ & $5.5085380$ & $15$ & $3.7314241$ & $25$ & $3.3679599$\\ \hline
$6$ & $5.0273441$ & $16$ & $3.6223695$ & $26$ & $3.3183350$\\ \hline
$7$ & $4.6151697$ & $17$ & $3.5254483$ & $27$ & $3.2873677$\\ \hline
$8$ & $4.2993763$ & $18$ & $3.4590869$ & $28$ & $3.2835481$\\ \hline
$9$ & $4.0978890$ & $19$ & $3.4312823$ & $29$ & $3.3051186$ \\ \hline
$10$ & $3.9989275$ & $20$ & $3.4359883$ & $30$ & $3.3414988$\\ \hline
\end{tabular}
\end{table}

We end this section with some figures and a table of the {\bf scaled}
probability density functions for the number of rolls of a fair die
until visiting the primes $k$ times for various $k$ values.
(Recall that the scaled version of a random variable $X$ with
expectation $\mu$ and variance $\sigma^2$ is $(X-\mu)/\sigma$).

\begin{table}[H]
\begin{tabular}{|c|c|c|c|c|}
\hline
$k$  & Expectation &  Standard Deviation  & Skewness &  Kurtosis \\
\hline
$20$ & $100.2648068$ & $22.6771846$ & $0.5351098$ & $3.4359883$ \\ \hline
$40$ & $229.8903783$ & $36.1271902$ & $0.3777949$ & $3.1278526$ \\ \hline
$60$ & $370.5241578$ & $46.0245135$ & $0.1406763$ & $2.6164507$\\ \hline
$80$ & $520.2899340$ & $57.8152360$ & $0.2910580$ & $2.9707515$\\ \hline
$100$ & $676.3153763$ & $65.2765933$ & $0.2230411$ & $3.0704308$\\ \hline
\end{tabular}
\end{table}

\begin{figure}[H]
     \centering
     \begin{subfigure}[b]{0.3\textwidth}
         \centering
         \includegraphics[width=\textwidth]{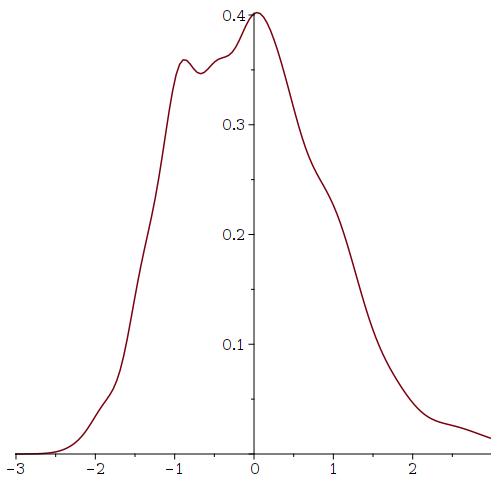}
         \caption*{$k=20$}
     \end{subfigure}
     \begin{subfigure}[b]{0.3\textwidth}
         \centering
         \includegraphics[width=\textwidth]{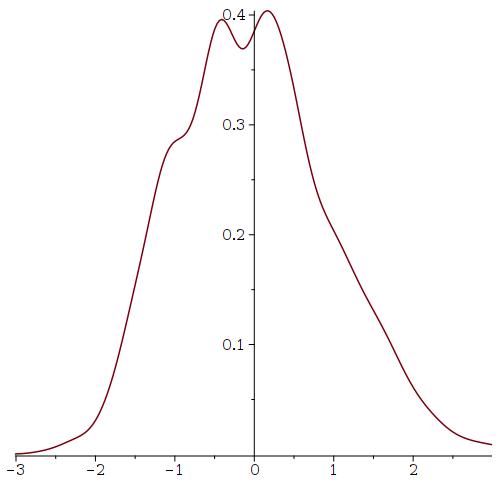}
         \caption*{$k=40$}
     \end{subfigure}

     \hspace{1in}

      \begin{subfigure}[b]{0.3\textwidth}
         \centering
         \includegraphics[width=\textwidth]{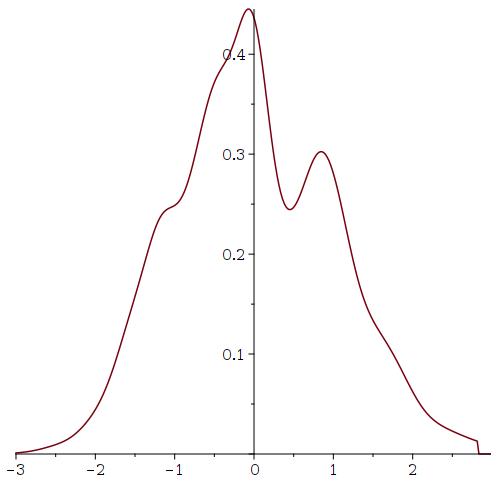}
         \caption*{$k=60$}
     \end{subfigure}
      \begin{subfigure}[b]{0.3\textwidth}
         \centering
         \includegraphics[width=\textwidth]{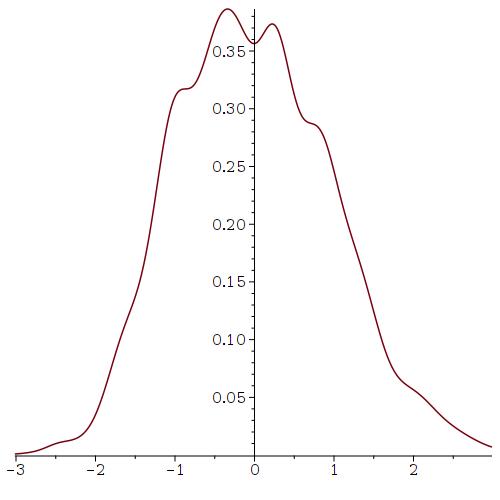}
         \caption*{$k=80$}
     \end{subfigure}
     \begin{subfigure}[b]{0.3\textwidth}
         \centering
         \includegraphics[width=\textwidth]{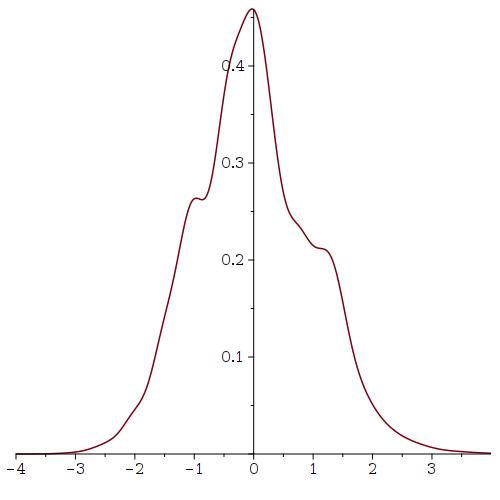}
         \caption*{$k=100$}
     \end{subfigure}
        \caption{Scaled probability density function for the
number of rolls of a fair die until visiting the primes $k$ times.}
        \label{fig:scaled prob pics}
\end{figure}
\vspace{0.2cm}

\noindent
Based on the available data above, the argument described
in the next section, and the known results about the
function $\pi(n)$ which is the number of primes that
do not exceed $n$, a possible guess
for a more precise expression for $E(L_k)$ may be
$k (\log  k + \log \log k +c_1)+c_2$.
This is also roughly consistent with the computational evidence.

\subsection{Asymptotic results}
In the next section we prove the following two results.
\begin{theo}
\label{t11}
For any fixed positive reals $\eps, \delta$
there exists $k_0=k_0(\eps, \delta)$ so that for all
$k>k_0$ the probability that $|L_k-k \log k| > \eps k \log k$
is smaller than $\delta$.
\end{theo}
\begin{theo}
\label{t12}
For any fixed $\eps>0$  and any $k> k_0(\eps)$, the
expected value of the random variable $L_k$  satisfies
$|E(L_k)- k \log k| < \eps k \log k$.
\end{theo}

\section{Proofs}
In all proofs we omit all floor and ceiling signs  whenever these are not
crucial, in order to simplify the presentation.

\label{sec:proofs}
\begin{lemma}
\label{l21}
There are fixed positive $C$ and
$\mu,\, 0<\mu<1$ so that the following holds.
Let $S=(d_1,d_2, \ldots)$ be a random sequence of independent
rolls of fair dice.
For any positive integer $x$, let $p(x)$ denote the probability
that $S$ hits $x$. Then $|p(x)-2/7| \leq C(1-\mu)^x$, that is,
as $x$ grows, $p(x)$ converges to the constant $2/7$ with an
exponential rate.
\end{lemma}
\begin{proof}
Define $p(-5)=p(-4)=p(-3)=p(-2)=p(-1)=0$, $p(0)=1$ and note
that for every $i \geq 1$,
$$
p(i)=\frac{1}{6}\sum_{j=1}^6 p( i-j).
$$
Indeed, $S$ hits $i$ if and only if the last number it hits
before $i$ is $i-j$ for some $j \in \{1, \ldots ,6\}$, and the
die rolled after that gives the value $j$. The probability of
this event for each specific value of $j$ is $p(i-j) \cdot (1/6)$,
providing the equation above. (Note that the definition of the initial
values is consistent with this reasoning, as before any dice rolls
the initial sum is $0$).  Thus, the sequence $(p(i))$
satisfies the homogeneous linear recurrence relation given above. The
characteristic polynomial of that is
$$
P(z)=z^6-\frac{1}{6}(z^5+z^4+z^3+z^2+z+1).
$$
One of the roots of this polynomial is $z=1$, and its multiplicity is
$1$ as the derivative of $P(z)$ does not vanish at $1$. It is also
easy to check that the absolute value of each of the other roots
$\lambda_j$, $2 \leq j \leq 6$
of $P(z)$ is at most $1-\mu$ for some absolute positive constant
$\mu,\, 0<\mu<1$. Therefore, there are constants $c_j$ so that
$$
p(i)=c_1 \cdot 1^i+\sum_{j=2}^6 c_j \lambda_j^i,
$$
implying that
$$
|p(i)-c_1| \leq C (1-\mu)^i
$$
for some absolute constant $C$. It remains to compute the value of
$c_1$. By the last estimate, for any positive $n$,
$$
|\sum_{i=1}^n p(i) -c_1 n| \leq C/(1-\mu).
$$
Note that the sum $\sum_{i=1}^n p(i)$ is the expected number of
integers in $[n]=\{1,2, \ldots,n\}$ hit by the sequence $S$.

For each fixed $f$, $d_1+d_2+\cdots+d_f$ is a sum of $f$ independent
identically distributed
random variables, each uniform on $\left\{1,2,\ldots,6\right\}$.
By the standard estimates for the distribution of sums of independent
bounded random variables, see., e.g., \cite{AS}, Theorem A.1.16,
this sum is very close
to $7f/2$ with high probability. Therefore for large $n$
the expectation considered above is $(1+o(1))(2/7)n$.
Dividing by $n$ and taking the limit as $n$ tends to infinity
shows that $c_1=2/7$, completing the proof.
\end{proof}

Note that the lemma above implies that there exists an absolute
positive constant $c$ so that for any (large) integer $g$ the
following holds:
\begin{equation}
\label{e21}
\mbox{For any}~~ x \geq c \log g-5, ~~
p(x)=\frac{2}{7}e^{\eps_1(x)},
1-p(x)=\frac{5}{7}e^{\eps_2(x)}~~ \mbox{where}~~
	|\eps_1(x)|< 1/g, |\eps_2(x)| \leq 1/g.
\end{equation}
It will be convenient to apply this estimate later.
\vspace{0.2cm}

Let $Y_m(S)$ denote the number of primes in $[m]=\{1,2,\ldots ,m\}$ hit by
$S$. In the next lemma we use the letters $H$ and $N$ to represent
"hit" and "not-hit", respectively.
\begin{lemma}
\label{l22}
For any sequence of integers $1 \leq x_1<x_2<\cdots<x_g$
that satisfy $x_1 \geq c \log g$ and
$x_{i+1}-x_i\geq c\log g$ for all $1 \leq i \leq g-1$,
where $c$ is the constant from (\ref{e21}),
and for every
$\nu \in \left\{H,N\right\}^{g}$ the following holds.
Let $h$ be the number of $H$ coordinates of $\nu$. Then,
\begin{align*}
&
P\left(S\,\,\, \text{hits}\,\,\, x_i\,\,\, \text{iff}\,\,\,
\nu_i=H\right)=
\left(\frac{2}{7}\right)^h\left(\frac{5}{7}\right)^{g-h}e^{\eps(\nu)},
\end{align*}
where $|\eps(\nu)|\leq 1.$
\end{lemma}
\begin{proof}
The probability of
the event $\left(S \,\,\, \text{hits}\,\,\, x_i\,\,\, \text{iff}\,\,\,
\nu_i=H\right)$ is a product of $g$ terms. The first term is the
probability that
$S$ hits $x_1$ (if $\nu_1=H$) or the probability that
$S$ does not hit $x_1$ (if  $\nu_1=N$). Note that since
$x_1 > c \log g$ this probability is
$\frac{2}{7}e^{\eps_1}$ in the first case and
$\frac{5}{7}e^{\eps_2}$ in the second case, where
both  $|\eps_1|$ and $|\eps_2|$ are at most $1/g$.

The second term in the product is the conditional probability that
$S$ hits $x_2$ (if $\nu_2=H$), or that it does not
hit $x_2$ (if $\nu_2=N$),
given the first value it hit in the interval
$x_1,x_1+1, \ldots, x_1+5$. If $\nu_1=H$, this first value
is $x_1$ itself, and then the probability to hit $x_2$
is exactly $p(x_2-x_1)$. If $\nu_1=N$, then this first value
is one of the $5$ possibilities $x_1+j$ for some $1 \leq j \leq 5$.
Subject to hitting $x_1+j$, the conditional probability to hit $x_2$
is exactly $p(x_2-x_1-j)$, which by the assumption on the
difference $x_2-x_1$, is very close to $\frac{2}{7}$.
By the law of total probability it follows that in any case
the conditional probability to hit $x_2$ is
$\frac{2}{7} e^{\eps'}$ and the conditional probability
not to hit it is $\frac{5}{7}e^{\eps"}$ where the absolute
value of $\eps'$ and of $\eps"$ is at most $1/g$.
Continuing in this manner we get a product of
$g$ terms, $h$ of which are very close to $2/7$ and
$g-h$ are very close to $5/7$, where the product of all
error terms $e^{\eps'''}$ is of the form $e^{\eps}$
for some $|\eps| \leq g \cdot (1/g) =1$. This completes
the proof of the lemma.
\end{proof}

\begin{prop}
\label{prop}
For any sequence $x_1<x_2<\cdots<x_n$ of positive integers
and any $a\geq \sqrt{n}\log(n)$
\begin{align*}
&
P\left(\Big|\# x_i\,\,\,\, \text{hit}\,\,\,-\frac{2}{7}n\Big |
\geq a\right)\leq e^{-c'\frac{a^2}{n\log(n)}},
\end{align*}
for some absolute positive constant $c'$.
\end{prop}
\begin{proof}
Split $x_1,\ldots,x_n$ into $c\log(n)$ subsequences,
where subsequence number $j$ consists of all $x_i$ with index
	$i \equiv j~ \mbox{mod}~(c \log n)$
where $c$ is the constant
from (\ref{e21}). Note that the difference between any two distinct
elements in the same subsequences is at least $c \log n$ and that
each of these subsequences  can contain at most one element
smaller than $c \log n$.
Each one of the subsequences
contains $r:=\frac{n}{c\log(n)}$ elements $x_i$.
In each subsequence, the probability to deviate in absolute
value from $\frac{2}{7}r$ hits  by more than $\frac{a}{c\log(n)}$
can be bounded by the
Chernoff’s bound for binomial distributions, up to
a factor of $e$. Indeed, Lemma \ref{l22} shows that
the contribution of each term does not exceed the
contribution of the corresponding term for the binomial random variable
with parameters $r$ and $2/7$
by more than a factor of $e$. Note that although each subsequence
may contain one element smaller than $c \log n$, the contribution of this
single element to the deviation is negligible and can be ignored.
Plugging in the standard bound, see, e.g.
\cite{AS}, Theorem A.1.16, we get that the probability of the
event considered is at most
$$
2e \cdot e^{-c^{\prime}\left(\frac{a}{c\log(n)}\right)^2/
\left(\frac{n}{c\log(n)}\right)} \leq
e^{-c^{\prime\prime}\frac{a^2}{n\log(n)}}
$$
for appropriate absolute constants $c'$, $c''$. Here we used the fact
that since $a$ is large the constant $2e$ can be swallowed by the choice
of $c''$.
Therefore, the probability to deviate in at least one of the
	subsequences by more than $a/(c \log n)$ is at most
$$ c\log(n)e^{-c^{\prime\prime}\frac{a^2}{n\log(n)}}
\leq e^{-c^{\prime \prime \prime}\frac{a^2}{n\log(n)}},
$$
where in the last inequality we used again the fact
that $a\geq\sqrt{n}\log(n)$.
\end{proof}

Recall that $L_k$ is the minimum $i$ so that $S$ hits $k$ primes
in the first $i$ steps.
\begin{coro}
	\label{c24}
	(1) If $\frac{2}{7}\pi(m_1)\leq k-a$ and
	$a \geq \sqrt{\pi(m_1)} \log (\pi(m_1))$,
then
\begin{align*}
&
P\left(Y_{m_1}\geq k \right)\leq
e^{-c^{\prime \prime \prime}\frac{a^2}{\pi(m_1)\log(\pi(m_1))}}.
\end{align*}
\bigskip

\noindent
(2) If $\frac{2}{7}\pi(m_2)\geq k+a$ and
	$a \geq \sqrt{\pi(m_2)} \log (\pi(m_2))$
then
\begin{align*}
&
P\left(Y_{m_2}\leq k \right)\leq e^{-c^{\prime
\prime \prime}\frac{a^2}{\pi(m_2)\log(\pi(m_2))}}.
\end{align*}
\end{coro}
\begin{proof}
(1) The event ${\displaystyle \left\{Y_{m_1}\geq k\right\}}$
means that the number of primes that are at most $m_1$ and are hit
by the infinite sequence of the initial sums of dice rolls is a least $k$.
Therefore, if $\frac{2}{7}\pi(m_1)\leq k-a$, we have
\begin{align*}
&
P\left(Y_{m_1}\geq k \right)=P\left(Y_{m_1}-\frac{2}{7}\pi(m_1)
\geq k-\frac{2}{7}\pi(m_1) \right)\leq P\left(\Big|Y_{m_1}-\frac{2}{7}
\pi(m_1)\Big|\geq a \right)\leq e^{-c^{\prime \prime \prime}
\frac{a^2}{\pi(m_1)\log(\pi(m_1))}},
\end{align*}
where the last inequality follows from Proposition
\ref{prop}.
\bigskip

\noindent
(2) Similarly, if $\frac{2}{7}\pi(m_2)\geq k+a$, we have
\begin{align*}
&
P\left(Y_{m_2}\leq k \right)=P\left(Y_{m_2}-\frac{2}{7}\pi(m_2)
\leq k-\frac{2}{7}\pi(m_2) \right)\leq P\left(Y_{m_2}-\frac{2}{7}
\pi(m_2)\leq -a \right)\\
&
\leq
P\left(\Big|Y_{m_2}-\frac{2}{7}\pi(m_2)\Big|\geq a \right)
\leq e^{-c^{\prime \prime \prime}\frac{a^2}{\pi(m_2)\log(\pi(m_2))}},
\end{align*}
where the last inequality follows from Proposition
\ref{prop}.
\end{proof}

\begin{coro}
	\label{c25}
	(1) For a given (large) $k$, let $m_1$ be the smallest integer
	so that
$$\pi(m_1)=\lfloor \frac{7}{2}(k -2\sqrt k \log k \rfloor.$$
	Then for any $i$ satisfying
	$\frac{7}{2}i\leq m_1-a$, where
	$a \geq 2 \sqrt{k} \log (k)$,
\begin{align*}
&
P\left(L_k\leq i\right)\leq P\left(d_1+\cdots+d_i\geq m_1\right)
+P\left(Y_{m_1}\geq k\right)\leq e^{-c^{\prime \prime \prime\prime}
\frac{a^2}{i}}+e^{-c^{\prime \prime \prime}
	\frac{k \log^2 k}{\pi(m_1)\log(\pi(m_1))}} \leq k^{-\alpha}
\end{align*}
for some absolute constant $\alpha>0$.
\bigskip

\noindent
	(2) For a given (large) $k$ and for $a \geq \sqrt k \log^2 k$
	let $m_2$  be the smallest integer so that
	$$\pi(m_2)=\lceil \frac{7}{2} (k+a) \rceil.$$
	Then for any $i$ satisfying
	$\frac{7}{2}i\geq m_2+b$, where $b \geq a$
\begin{align*}
&
P\left(L_k\geq i\right)\leq P\left(d_1+\cdots+d_i\leq
m_2\right)+P\left(Y_{m_2}\leq k\right)\leq e^{-c^{\prime
\prime \prime\prime}\frac{b^2}{i}}+e^{-c^{\prime \prime
\prime}\frac{a^2}{\pi(m_2)\log(\pi(m_2))}}.
\end{align*}
\end{coro}

\begin{proof}
(1) If both events ${\displaystyle \left\{d_1+
\cdots+d_i\geq m_1\right\}}$ and
${\displaystyle \left\{Y_{m_1}\geq k\right\}}$ do not occur,
then the event ${\displaystyle \left\{L_k\leq
i\right\}}$ does not occur.
Therefore, for $\frac{7}{2}i\leq m_1-a$ we have
\begin{align*}
&
P\left(L_k\leq i\right)\leq P\left(d_1+\cdots+d_i\geq m_1\right)+
P\left(Y_{m_1}\geq k\right)\leq P\left(d_1+\cdots+d_i-\frac{7}{2}i
\geq a\right)+P\left(Y_{m_1}\geq k\right)\\
&
\leq
 e^{-c^{\prime \prime \prime\prime}\frac{a^2}{i}}+
e^{-c^{\prime \prime \prime}\frac{k \log^2 k}{\pi(m_1)\log(\pi(m_1))}},
\end{align*}
where the last inequality follows from
Chernoff's bound and the first part of Corollary \ref{c24}.
	Note that here $2 \sqrt k \log k \geq \sqrt{\pi(m_1)}
	\log (\pi(m_1))$ and therefore the corollary can be applied.
\bigskip

(2) Similarly, if both events ${\displaystyle \left\{d_1+\cdots
+d_i\leq m_2\right\}}$ and
${\displaystyle \left\{Y_{m_2}\leq k\right\}}$ do not occur,
then the event ${\displaystyle
\left\{L_k > i\right\}}$ does not occur.
Therefore, for $\frac{7}{2}i\geq m_2+b$, we have
\begin{align*}
&
P\left(L_k\geq i\right)\leq P\left(d_1+\cdots+
d_i-\frac{7}{2}i\leq -b\right)+P\left(Y_{m_2}\leq k\right)
\leq e^{-c^{\prime \prime \prime\prime}\frac{b^2}{i}}+
e^{-c^{\prime \prime \prime}\frac{a^2}{\pi(m_2)\log(\pi(m_2))}},
\end{align*}
where the last inequality follows again from
Chernoff's bound and the second part of Corollary \ref{c24}.
Indeed the corollary can be applied since it is not difficult
to check that for large $k$ and any $a \geq \sqrt k \log^2 k$,
	$$a \geq \sqrt {\pi(m_2)} \log (\pi(m_2))=
	\sqrt{\lceil \frac{7}{2}(k+a) \rceil} \log
	(\lceil \frac{7}{2}(k+a) \rceil). $$
\end{proof}
\bigskip
\vspace{0.2cm}

\noindent
{\bf Proof of Theorem \ref{t11}:}\,
Note that by the Prime Number Theorem in the first part of
Corollary \ref{c25}, $$m_1=(\frac{7}{2}+o(1)) k \log k.$$
Taking $a=2 \sqrt k \log k$ and letting $i_1$ be the largest integer
so that $\frac{7}{2} i \leq m_1 -a$
it follows from this first part that $i_1=(1+o(1)) k \log k$
and that the probability that $L_k$ is smaller than
$i_1$ is smaller than some negative power of $k$, that is, tends
to $0$ as $k$ tends to infinity.

Similarly, substituting in the second part of the corollary
$a = b=  \sqrt k \log^2 k$ and letting $i_2$ be the
smallest integer so that $\frac{7}{2}i \geq m_2 +a$ it is easy to see
that $i_2$ is also $(1+o(1))k \log k$
(since
$$m_2=(\frac{7}{2}+o(1)) k \log k,$$
by the Prime Number Theorem).  By the second
part of the corollary the probability that $L_k$ is larger than
$i_2$ is smaller than any fixed negative power of $k$, and hence
tends to $0$ as $k$ tends to infinity. Therefore
$L_k$ is $(1+o(1)) k \log k$ with probability tending to $1$ as
$k$ tends to infinity, completing the proof of
the theorem. \hfill $\Box$
\vspace{0.2cm}

\noindent
{\bf Proof of Theorem \ref{t12}:}\,
The expectation of $L_k$ is the sum over all positive integers
$i$, of the probabilities $P(L_k \geq i)$. Taking $a= \sqrt k \log^2 k$
and defining $m_1$ and $m_2$ as before we break this sum into
three parts,
$$
S_1=\sum_{i:\frac{7}{2}i\leq m_1-a }P\left(L_k\geq i\right),
$$
$$
S_2=\sum_{i: \frac{7}{2}i \geq m_2+a}P\left(L_k\geq i\right),
$$
and
$$
S_3=\sum_{i: m_1-a < \frac{7}{2}i < m_2+a}P\left(L_k\geq i\right).
$$
By the first part of Corollary \ref{c25} each summand
in the first sum $S_1$ is
$1-o(1)$ and therefore $S_1=(1+o(1)) k \log k$, as the number
of summands is $(1+o(1)) k \log k$, since $m_1=(\frac{7}{2} +o(1))
k \log k.$
By the second part of the corollary (applied to an
appropriately chosen
sequence of $a, m_2$ and $b$)
it is not difficult to check
that the infinite
sum $S_2$ is only $o(1)$. Indeed, it is possible, for example,
to choose $a_0=\sqrt k \log^2 k$ and $a_j=jk$ for all $j \geq 1$.
The corresponding value $m_{2,j}$ of $m_2$ for each $a_j$ is defined as the
smallest integer satisfying $\pi(m_{2,j})=\lceil \frac{7}{2} (k+a_j) \rceil$.
Taking $b_j = a_j$ we can apply the estimate in the second part of
the corollary to all values of $i$ satisfying
$m_{2,j}+a_j \leq \frac{7}{2}i < m_{2,j+1}+a_{j+1}$.
The sum of the probabilities $P(L_k \geq i)$ for these values of $i$
is thus at most $k e^{-\Omega(\log^3 k)}$ for $j=0$, and at most
$k e^{-\Omega(j k/\log (jk))}$ for each $j \geq 1$.
The sum of all these quantities is smaller than any fixed negative
power of $k$, and is therefore $o(1)$, as needed.

The sum $S_3$ is a sum of at most $m_2-m_1+2a$ terms, and each
of them is at most $1$, implying that $0 \leq S_3 \leq
m_2-m_1+2a=o(k \log k)$, since both $m_1$ and $m_2$ are
$(\frac{7}{2}+o(1)) k \log k$, and $2a=O(\sqrt k \log^2 k)$.
This completes the proof of the theorem. \hfill $\Box$

\section{Concluding remarks and extensions}
\begin{itemize}
\item
{\bf Extensions for biased $r$-sided dice and arbitrary
subsets of the integers.}  The proofs in the previous section
use very little of the specific properties of the primes and
the specific distribution of each $d_i$. It is easy to
extend the result to any $r$-sided dice with an arbitrary discrete
distribution on $[r]$ in which the values obtained with positive
probabilities do not have any nontrivial common divisor. The constants
$3.5$ and $2/7$ will then have to be replaced by the expectation of
the random variable $d_i$ and by its reciprocal, respectively.
It is interesting to note that while for different dice the expectation
of $L_k$ for small values of $k$ can be very different from the
corresponding expectation for a standard fair die, once the die
is fixed, for large $k$ the expectation
is always $(1+o(1))k \log k$, where the $o(1)$-term tends
to $0$ as $k$ tends to infinity.

It is also possible to replace the primes by an arbitrary subset $T$
of the
positive integers, and repeat the arguments to analyze the
corresponding random variable for this case, replacing the Prime
Number Theorem by the counting function of $T$. We omit the details.
\item
{\bf Heuristic suggestion for a more precise expression for
$E(L_k)$.} If we substitute for $\pi(n)$ its approximation
$n/\log n$ and repeat the analysis described here
with this approximation, the more precise
value for the expectation $E(L_k)$ that follows
is $k (\log  k + \log \log k +O(1))$. Since at the
beginning there are some fluctuations, we tried
to add another constant and consider an expression
of the form
$k (\log  k + \log \log k +c_1)+c_2$. Choosing $c_1$ and $c_2$ that
provide the best fit for our (limited and therefore maybe overfitted) computational
evidence we obtained the heuristic expression $f(k)=k (\log  k + \log \log k +0.543)+8.953.$
For the record, here are the ratios of $E[L_k]/f(k)$ for $k=20,40,60,80,100$, respectively:

$0.9861651120, 0.9976101939, 0.9966486957, 0.998338113, 0.9997448512$.

One can also replace $n/\log n$ by the more precise
approximation $Li(n)$ for $\pi(n)$, but the difference
between these two estimates does not change the
expression obtained for $E(L_k)$ in a significant way.

\end{itemize}

\noindent
{\bf Acknowledgment}
Noga Alon is supported in part by NSF grant DMS-2154082.
Yaakov Malinovsky is supported in part by BSF grant 2020063.
Lucy Martinez is supported by the NSF Graduate Research Fellowship Program under Grant No. 2233066.
YM expresses gratitude to Stanislav Molchanov for suggesting that the hits on squares asymptotically follow a success probability of $1/3.5$.

\end{document}